\documentclass[12pt]{amsart}
\usepackage{amsmath}
\usepackage{amsfonts}
\usepackage{amssymb}
\usepackage[]{amsmath, amsthm, amsfonts, amssymb, epsfig}
\numberwithin{equation}{section}

\begin{document}
\title[Khavinson-Shapiro Conjecture for the Bergman Projection]{The Khavinson-Shapiro Conjecture for the Bergman Projection in one and Several Complex Variables}

\author{Alan R. Legg}

\address{Mathematics Department, Purdue University, West Lafayette,
IN  47907}

\email{arlegg@purdue.edu}
\thanks{Research supported by the NSF Analysis and Cyber-enabled
Discovery and Innovation programs, grant DMS~1001701}

\begin{abstract}
We reveal a complex analogue to a result about polynomial solutions to the Dirichlet Problem on ellipsoids in $\mathbb{R}^n$ by showing that the Bergman projection on any ellipsoid in $\mathbb{C}^n$ is such that the projection of any polynomial function of degree at most $N$ is a holomorphic polynomial function of degree at most $N$.  The discussion is motivated by a connection between the Bergman projection and the Khavinson-Shapiro conjecture in $\mathbb{C}$. We also relate the Khavinson-Shapiro conjecture to polyharmonic Bergman projections in $\mathbb{R}^n$ by showing that these projections take polynomials to polynomials on ellipsoids.
\end{abstract}

\maketitle

\theoremstyle{plain}

\newtheorem {thm}{Theorem}[section]
\newtheorem {lem}[thm]{Lemma}
\newtheorem {prop} [thm] {Proposition}
\section{Introduction and Notation}

An intriguing connection between the Dirichlet problem and the Bergman projection can be found via the Khavinson-Shapiro conjecture, which in one formulation posits that ellipsoids are the only domains on which the Dirichlet problem solution operator for the Laplacian takes polynomial boundary data to polynomials (cf. Sections 2 and 5 of \cite{KS}).  If we modify the Khavinson-Shapiro conjecture by replacing the Dirichlet problem solution operator with the Bergman projection, then in the special case of smooth bounded planar domains, we will see below that we actually obtain a statement equivalent to the original Khavinson-Shapiro conjecture.  This observation is the starting point here for a consideration of the Bergman projections of polynomial functions on ellipsoids in more than one complex variable.  
  
   For the case of the Laplacian on real space, it is a fact that the Dirichlet problem solution operator on an ellipsoid takes polynomial boundary data into harmonic polynomials.  That is to say, whenever the boundary values of a polynomial are given on an ellipsoid, it follows that the harmonic function on the ellipsoid which attains the same boundary values is also a polynomial.  Furthermore, the degree of this harmonic polynomial does not exceed the degree of the polynomial whose boundary data were given.  The result can be obtained very elegantly by the use of a linear map from the set of polynomials to itself (the so-called ``Fischer Map"). For a good treatment of the details, see Proposition 1 of \cite{KL}; another exposition along the same lines can be found in Sections 1 and 2 of \cite{Baker}.  
  
  Employing an argument in the same spirit, we establish in Section 3 an analytic analogue for ellipsoids in $\mathbb{R}^{2n} \sim \mathbb{C}^n$, showing that the Bergman projection on ellipsoids maps polynomials to polynomials. Even more specifically, we show that the Bergman projection of any polynomial function on an ellipsoid is a holomorphic polynomial of equal or lesser degree. 
  
  Further background for questions related to the Khavinson-Shapiro conjecture can be found in the article \cite{KL1}. For more works pertaining to the use of Fischer maps and related machinery in partial differential equations, we direct the reader to the papers \cite{Sh1}, \cite{LR}, \cite{Re}, \cite{R}.

   Recall for the sake of precision  that given a domain $\Omega \subset \mathbb{C}^n$, the Bergman projection $B: \thinspace L^2(\Omega) \rightarrow H^2(\Omega)$ on $\Omega$ is the orthogonal projection from $L^2(\Omega)$ onto its subspace $H^2(\Omega)$ consisting of holomorphic functions which are square-integrable with respect to Lebesgue measure. Here we are employing the usual inner product on $L^2(\Omega)$; namely, given $f,g \in L^2(\Omega)$, their inner product is $\langle f, g \rangle = \int_{\Omega}^{}f\bar{g}dV$, where $dV$ is Lebesgue measure.

As a matter of notation, let $z=(z_1, z_2,  \dots, z_n)$ denote the coordinates of $\mathbb{C}^n$, and let $x_j=Re( z_j)$ and $\thinspace y_j=Im( z_j),\thinspace j=1,\thinspace 2,\dots,n$ denote the real coordinates on $\Omega$. We further say $x=(x_1, x_2, \dots, x_n)$ and $ \thinspace y=(y_1, y_2, \dots, y_n)$, and we let $\alpha, \beta, \gamma$ stand for $n$-dimensional multi-indices.  Then, using the usual multi-index notation, we define for each nonnegative integer $N$ the following sets of functions on $\Omega$:
\begin{equation}
\label{Pdef}
  P_N= \{\sum_{|\alpha|+|\beta| \leq N}c_{\alpha,\beta}x^\alpha y^{\beta}\thinspace :\thinspace c_{\alpha, \beta} \in \mathbb{C} \},
  \end{equation}
the set of (not-necessarily-holomorphic) complex-valued polynomial functions of degree at most $N$, and
  \begin{equation}
  \label{Hdef}
   HP_N= \{\sum_{|\gamma|\leq N}^{} d_{\gamma}z^{\gamma} \thinspace : \thinspace d_{\gamma} \in \mathbb{C} \},
   \end{equation}
the set of holomorphic polynomial functions of degree at most N. The content of our main result, then, is that on any ellipsoid, $B(P_N)=HP_N$ for each nonnegative integer $N$.

The close similarity to the case of the Dirichlet problem solution operator on an ellipsoid naturally brings us back to the consideration of a `Khavinson-Shapiro'-type conjecture for the Bergman projection in several dimensions; i.e., we may ask whether multi-dimensional ellipsoids are at all characterized by the property that the Bergman projection maps polynomials to polynomials.  In Section 4, we work toward a generally negative answer to the question, exhibiting in this case non-ellipsoidal domains on which the Bergman projection maps polynomials to polynomials. 

Finally, we return in Section 5 to the linear-algebra-style proof used in Section 3 to show that the polyharmonic Bergman projections take polynomials to polynomials on ellipsoids in real space.  This serves to open the possibility of a hierarchy of Khavinson-Shapiro conjectures.

\section{The Bergman Projection and the Khavinson-Shapiro Conjecture in $\mathbb{C}$}

As motivation for considering the Bergman projection as it acts on polynomials in ellipsoidal domains, we first consider the case of the plane.  In this case, it is true that the Bergman projection takes polynomials to polynomials on ellipses, but by simple calculations we show something a bit stronger, which is related to the Khavinson-Shapiro conjecture.  

It turns out, as presented in the next proposition, that for smooth bounded domains in the plane, the Dirichlet problem solution operator takes polynomials to polynomials if and only if the Bergman projection takes polynomials to polynomials.  Thus the Khavinson-Shapiro conjecture is equivalent in this case to the analogous formulation involving the Bergman projection instead of the Dirichlet problem solution operator.  To see this requires the fact that holomorphy and harmonicity in the plane are related by differentiation (if $f$ is harmonic, then $\frac{\partial f}{\partial z}$ is holomorphic).

\begin{prop}
Let $\Omega \subset \mathbb{C}$ be a $\mathcal{C}^{\infty}$-smooth bounded domain.  Then the Bergman projection of $\Omega$ maps polynomials to polynomials if and only if the Dirichlet problem solution operator takes polynomial boundary data to polynomials.  
\end{prop}
\begin{proof}
First assume that the Bergman projection maps polynomials to polynomials, and let $Q(z,\thinspace \bar{z})$ be a real-valued polynomial function on $\Omega$.  By Havin's Lemma (see, e.g., pages 26 and 82 of \cite{Sh}), we then have the orthogonal decomposition
$\frac{\partial Q}{\partial z}=p(z) + \frac{\partial \varphi}{\partial z}$, where $p$ is the Bergman projection of $\frac{\partial Q}{\partial z}$, and $\varphi$ is $\mathcal{C}^{\infty}$-smooth up to the boundary of $\Omega$, and vanishes on the boundary of $\Omega$.  By hypothesis, $p$ is a holomorphic polynomial; and by formal antidifferentiation let $P$ be a holomorphic polynomial such that $P'=p$.  Then we have that $\frac{\partial}{\partial z}(Q-P-\varphi)=0$.  Hence the function being differentiated on the left is anti-holomorphic, say $Q-P-\varphi = \bar{H}$, where $H \in H^2(\Omega)$. But now notice that $Q-\varphi$ is harmonic and equal to $Q$ on the boundary, and so is the solution to the Dirichlet problem with boundary data $Q$.  Since $Q$ is real-valued, so is the harmonic extension of its boundary values, and so $P+\bar{H}=\bar{P}+H$.  But this means that $P-H=\bar{P}-\bar{H}$, and so $P-H$ must be constant (it is both holomorphic and antiholomorphic).  Hence $H$ is a polynomial.  But this means that the solution to the Dirichlet problem with boundary data $Q$, is a polynomial.  Now by linearity and breaking into real and imaginary parts, we see that the Dirichlet solution is polynomial for any complex-valued polynomial boundary data.

Conversely, assume that the Dirichlet solution is polynomial whenever the boundary data of a polynomial is given on $bd \Omega$.  Then, let $q(z,\thinspace \bar{z})$ be any polynomial.  By formal antidifferentiation in $z$, let $Q$ be a polynomial function such that $\frac{\partial Q}{\partial z}=q$.  Let $p$ be the Bergman projection of $q$.  Just as above, there exists $\varphi$ smooth up to the boundary and vanishing on $bd \Omega$ such that we have the orthogonal decomposition $\frac{\partial Q}{\partial z}=p + \frac{\partial \varphi}{\partial z}$.  Differentiate this equation with respect to $\bar{z}$ to conclude that $\Delta Q = \Delta \varphi$.  Hence $Q-\varphi$ is harmonic, and has the same boundary values as $Q$.  Hence it is the Dirichlet solution for boundary data $Q$, and so by hypothesis $Q-\varphi$ is a polynomial, and so $\frac{\partial \varphi}{\partial z}$ is also a polynomial.  But now, returning to the relation $q=p+\frac{\partial \varphi}{\partial z}$, we have that $p$ is in fact a polynomial.
\end{proof}

Hence for bounded smooth planar domains, the Khavinson-Shapiro conjecture can be rephrased to the effect that ellipses should be the only smooth bounded planar domains on which the Bergman projection maps polynomials to polynomials.  From here on, we will investigate the situation in more than one variable. 

  We emphasize that, while we were able to employ the fact that holomorphy and harmonicity are related simply by a differentiation in $\mathbb{C}$,  this is not so in more than one complex variable.  For this reason, we of course expect that any relationship between the behaviors of the Dirichlet solution operator and the Bergman projection in several variables will be more complicated than in the planar case.  Nevertheless, a strong similarity will be found. We will find that the Bergman projection continues to take polynomials to polynomials on ellipsoidal domains, but that other classes of domains have the same behavior.

\section{The Bergman Projection of Polynomials on Ellipsoids }

To proceed with a consideration of matters in more than one dimension, we will need to have a few pertinent facts at our disposal, which are collected here for reference. Note particularly that our setting will be applicable to all ellipsoids in $\mathbb{C}^n$, not just complex ellipsoids.
 
Any ellipsoid is by definition a quadric; that is, given an ellipsoid $\Omega \subset \mathbb{C}^n \sim \mathbb{R}^{2n}$, there exists a polynomial function $r(x_1,x_2, \dots, x_n, y_1, y_2, \dots, y_n) \thinspace$ on $\mathbb{C}^n$ such that the degree of $r$ is equal to $2$, $r$ vanishes on the boundary of $\Omega$, and
\begin{equation}
\label{quadric}
\Omega= \{z \in \mathbb{C}^n \thinspace : \thinspace r(x_1, x_2, \dots, x_n, y_1, y_2, \dots, y_n)< 0 \}.
\end{equation}

In addition, letting $H^2(\Omega)^\perp$ be the orthogonal subspace to $H^2(\Omega)$ in $L^2(\Omega)$, we have the following: if $\omega$ is any smooth $(0,1)$-form on $\Omega$ which extends smoothly to the boundary of $\Omega$ and vanishes on the boundary of $\Omega$, then 
\begin{equation}
\label{perp}
\vartheta \omega \in H^2(\Omega)^\perp,
\end{equation}
where $\vartheta$ is the formal adjoint to the $\overline{\partial}$ operator (see Section 3 of \cite{Bell1}).

 Finally, we point out that although $\vartheta$ and $\overline{\partial}$ are merely formal adjoints, they act as true Hilbert space adjoints for certain pairs of forms or functions.  Among these cases is that of the inner product of two functions, each of which is smooth up to the boundary of $\Omega$, and one of which is of the form $\vartheta \omega$, where $\omega$ is a smooth $(0,1)$-form which extends smoothly to the boundary of $\Omega$ and vanishes on the boundary of $\Omega$. To be precise, if the other function in the inner product is $f$, then in this case we may write 
 \begin{equation}
 \label{adjoint}
 \langle \vartheta \omega, f \rangle = \langle \omega, \overline{\partial}f \rangle,
 \end{equation}
  where we use $\langle \cdot \thinspace, \cdot \rangle$ to denote both the usual $L^2$ inner product on functions, and the $L^2$ inner product on $(0,1)$-forms, defined as the sum of the inner products of the respective component functions of the forms involved. A detailed account of the adjointness properties of $\vartheta$ and $\overline{\partial}$ on smooth bounded domains can be found in \cite{FK}.

We are now ready to state our main result:

\begin{thm}
\label{main}
 Suppose that $\Omega \subset \mathbb{C}^n$ is an ellipsoid and, as in (\ref{Pdef}) and (\ref{Hdef}), let $P_N$ and $HP_N$ be, respectively, the space of complex-valued polynomial functions on $\Omega$ of degree at most N, and the space of holomorphic polynomial functions of degree at most N.  Denote by $B$ the Bergman projection on $\Omega$. Then for each nonnegative integer $N$, $B(P_N)=HP_N$.
 \end{thm} 

\begin{proof} The inclusion $HP_N \subset B(P_N)$ is clear, since $HP_N$ is a subset of $P_N$ which is invariate under $B$. 

 Considering $P_N$ and $HP_N$ as finite-dimensional complex vector spaces, and noting that $HP_N$ is a subspace of $P_N$, form the quotient vector space $P_N/HP_N$.  The idea of the proof of the inclusion $B(P_N) \subset HP_N$ will be to exploit a certain vector space isomorphism of $P_N/HP_N$ with itself to obtain an orthogonal decomposition for elements of $P_N$.  

To this end, let $r$ be a degree-2 defining polynomial for $\Omega$ as in (\ref{quadric}), so that $r < 0$ on $\Omega$ and $r|_{\partial \Omega} = 0$, and define the map $\varphi : P_N \rightarrow P_N$ according to the formula
  \[\varphi (p) = \vartheta r \overline{ \partial}p \quad \text{ for each $p \in P_N$}.  \] 
  That $\varphi$ does in fact preserve degree is a consequence of the fact that \[ \vartheta r \overline{ \partial}p = -\sum _{j=1}^n \frac{\partial}{\partial z_j}(r \frac{\partial p}{\partial \bar{z}_j}), \] and each term of this sum has degree at most $N$; for $p$ itself has degree at most $N$, and each differentiation reduces the degree by at least $1$, while multiplying by $r$ increases the degree by at most $2$. It is clear, moreover, that $\varphi$ is complex-linear; and if it happens that $p \in HP_N$, then $\overline{\partial} p = 0$, so that $\varphi (p) =0$.
   Hence $\varphi$ descends to a linear mapping $\tilde{\varphi} : P_N/HP_N \rightarrow P_N/HP_N $ according to $\tilde{\varphi}([p])=[\varphi(p)]\quad \text{for each}\quad [p] \in P_N/HP_N,$ where $[\thinspace \cdot \thinspace ]$ denotes equivalence class.  In fact, $\tilde{\varphi}$ is injective, as we now show. 
  
   Assume for the moment that $\tilde{\varphi} ([p])=[0]$.  In this case $\varphi (p)$ is equivalent to $0$ modulo $HP_N$, and so there exists $h \in HP_N$ such that $\vartheta r \overline{\partial}p=h.$
  However, $r \overline{\partial} p$ is a $(0,1)$-form on $\Omega$, smooth up to boundary of $\Omega$, which vanishes on the boundary of $\Omega$;  consequently, (\ref{perp}) gives that $\vartheta r \overline{\partial}p \in \mathnormal{H}^2(\Omega)^{\perp}$.  And now, since $h \in \mathnormal{H}^2(\Omega), $ we see that $\vartheta r \overline{\partial}p = 0$, and we may calculate: \[ 0=\langle - \vartheta r \overline{\partial}p, p \rangle \\ =\langle -r \overline{\partial}p, \overline{\partial}p \rangle \]  where in the second equality the use of the adjointess of $\vartheta$ and $\overline{\partial}$ is justified since $r \overline{\partial}p$ vanishes on the boundary (cf (\ref{adjoint})) above). Owing to the fact that $-r >0$ on $\Omega$ , we have demonstrated that the weighted $L^2$ norm of the $(0,1)$-form $\overline{\partial}p$ against a positive measure on $\Omega$ arising from a smooth function is $0$, which in turn implies that $\overline{\partial}p \equiv 0$, so $p$ is holomorphic and $[p]=[0]$.  So indeed $\tilde{\varphi}$ is injective.
  
  Now, $\tilde{\varphi}$ must also be surjective, being an injective linear map from a finite-dimensional vector space into a vector space of equal dimension. The surjectivity of $\tilde{\varphi}$ will provide us with the orthogonal decomposition we require to identify the Bergman projections of the elements of $P_N$.  
  
  Given any polynomial function $P \in P_N$, there exists a polynomial function $Q \in P_N$ such that $[P]=\tilde{\varphi}([Q])$; or, what is the same, there must exist a holomorphic polynomial function $H \in HP_N$ such that
  \begin{equation}
  \label{decomp}
  P=\vartheta r \overline{\partial}Q + H.
  \end{equation}
  Notice now, though, that $\vartheta r \overline{\partial}Q \in \mathnormal{H}^2(\Omega)^\perp$ by (\ref{perp}), and since $H \in \mathnormal{H}^2(\Omega),$ (\ref{decomp}) is in fact an orthogonal decomposition of $P$, and so we must have that $BP = H \in HP_N$. Thus $B(P_N) \subset HP_N$, and we are finished.
  \end{proof}
  
  \section{Other Domains on which the Bergman Projection Maps Polynomials to Polynomials}
  In response to the Khavinson-Shapiro-type question of how well ellipsoids may be characterized by the property that polynomials are mapped to polynomials under the Bergman projection, we provide here examples of other domains exhibiting the same property.
  
  \subsection{Bounded Circular Domains} 
  
  Let $R \subset \mathbb{C}^n$ be a bounded circular domain containing the origin, and let  $K(z, w)$ be the Bergman kernel function of $R$.  For convenience, for each multi-index $\alpha$ define $K_{0}^{\alpha}(z)=\frac{\partial ^\alpha K(z,w)}{\partial \bar{w}^\alpha}|_{w=0}$. As discussed in \cite{Bell2}, the function $K_{0}^{\alpha}$ is such that for each $f \in H^2(R)$,
  \begin{equation}
  \label{diffrepro} 
   \langle f, K_{0}^{\alpha}  \rangle = \frac{\partial ^\alpha f}{\partial z^\alpha}(0).
  \end{equation}
  Note that $K_0^\alpha$ is the unique such function in $H^2(R)$, being the integral kernel function guaranteed by the Riesz representation theorem for point evaluation of the $\alpha$- derivative at zero for functions in $H^2(R)$. 
  
  Since $R$ is a bounded circular domain containing $0$, it follows from \cite{Bell2} that the linear span of the $K_0^\alpha$ as $\alpha$ ranges over all multi-indices is identical to the set of holomorphic polynomial functions on $R$.  Even more precisely, we have that, given a particular multi-index $\alpha$, the set of homogeneous holomorphic polynomials of degree $|\alpha|$ is identical to the linear span of $ \{K_0^\gamma \thinspace : \thinspace |\gamma|=|\alpha| \}$.  
  
  With these preliminaries in place, we can show the following:
  
  \begin{thm}
  \label{circ}
  If $R \subset \mathbb{C}^n$ is a bounded circular domain containing the origin and $P_N$, $HP_N$ are as in (\ref{Pdef}) and (\ref{Hdef}), and if $B$ is the Bergman Projection on $R$, then $B(P_N)=HP_N$ for each non-negative integer $N$.
  \end{thm}
  \begin{proof}
  As in the proof of Theorem \ref{main}, it is easy to see that $HP_N \subset P_N$.
  
  For the reverse inclusion, by linearity it suffices to prove that for each pair of multi-indices $\alpha$, $\beta$ such that $|\alpha|+|\beta|=N$, $B(z^{\alpha}\bar{z}^\beta) \in HP_N$.  By the comments preceding the statement of the current theorem, we may calculate as follows:
  \[\begin{split}
  \langle f, B(z^\alpha \bar{z}^\beta) \rangle = \langle f, z^\alpha \bar{z}^\beta \rangle = \langle fz^\beta, z^\alpha \rangle = \\ \langle f z^\beta, \sum_{|\gamma|=|\alpha|}^{}c_\gamma K_0^\gamma \rangle = \sum_{|\gamma|=|\alpha|}^{}c_\gamma \frac{\partial ^\gamma (fz^\beta)}{\partial z^\gamma}|_{z=0}
  ,
  \end{split}\]
 where the $c_\gamma$ are constants depending only on $\alpha$. The sum on the far right can be simplified by the product rule, so that we get for some constants $d_\gamma$ which depend only on $\alpha$ and $\beta$,
 \begin{equation}
 \label{monprod}
 \langle f, B(z^\alpha \bar{z}^\beta) \rangle = \sum_{|\gamma| \leq |\alpha|}^{} d_\gamma \frac{\partial ^\gamma f}{\partial z^\gamma}|_{z=0}.
 \end{equation}
 Again by comments above, the sum on the right is equal to the inner product \[ \langle f, \sum_{|\gamma| \leq |\alpha|} d_\gamma K_0^\gamma \rangle, \]
 and the right member of this inner product is a holomorphic polynomial $H$ of degree at most $|\alpha| \leq N$. 
 
 Thus, we have found a polynomial $H \in HP_N$ such that $ \langle f, B(z^\alpha \bar{z}^\beta) \rangle = \langle f, H \rangle $ for every $f \in H^2(R)$; and since $H$ and $B(z^\alpha \bar{z}^\beta)$ are themselves in $H^2(R)$, it follows that $B(z^\alpha \bar{z}^\beta)=H$.
  \end{proof}
 
 Note that the only domains satisfying the hypotheses of Theorem \ref{circ} when $n=1$ are discs centered at the origin.  When $n>1$, Theorem \ref{circ} includes the case of `complex ellipsoids,' which have a defining polynomial as in (\ref{quadric}) of the form $r=-1 + \sum_{j=1}^{n}a_j |z_j|^2$, the $a_j, \thinspace j=1,2, \dots, n$ being positive real numbers.  For other ellipsoids, we must appeal to Theorem \ref{main}.
 
 \subsection{Images under certain biholomorphisms}  
 
 Using the transformation formula for the Bergman projection under biholomorphic mappings, we can show that under suitable biholomorphisms, the property that the Bergman projection maps polynomials to polynomials is preserved.  However, in this case we must relax the degree-preserving requirement that $B(P_N)=HP_N$.  Although we will not investigate the possible effects of this relaxation here, it is interesting to note that it does have meaningful consequences in the plane \cite{CS}.
 
 Let $\Omega$ and $V$ be domains in $\mathbb{C}^n$ and $f:\thinspace \Omega \rightarrow V$  a biholomorphism between them, and let $u=det(f')$ be the complex Jacobian determinant of $f$. Then, if $g \in L^2(V)$, it follows that $u \cdot g \circ f \in L^2(\Omega)$ and 
 \begin{equation}
 \label{bergtrans}
 B_\Omega (u \cdot g \circ f)=u \cdot (B_Vg)\circ f,
 \end{equation}
 where $B_\Omega$ and $B_V$ are the Bergman projections on $\Omega$ and $V,$ respectively. (cf. Ch. 3 Sec. 2 of \cite{Bergman})
 
  Using the notation of (\ref{Pdef}), let $P=\cup_{N \geq 0}P_N$ be the set of all polynomial functions on $\mathbb{C}^n$, and let $HP$ be its subset consisting of all holomorphic polynomial functions.  In what follows, we shall view the elements of $P$ and $HP$ as functions either  on $\Omega$ or on $V$, but no confusion on this point will arise in this context.  In addition, by a `polynomial biholomorphic mapping' we will mean a biholomorphic mapping whose component functions are polynomials. A polynomial biholomorphic mapping `with polynomial inverse' will be a polynomial biholomorphic mapping whose inverse mapping is also a polynomial biholomorphic mapping.

 \begin{thm}
 \label{bihol}
 Let $\Omega,\thinspace V \subset \mathbb{C}^n$ be domains, and let $B_\Omega$ and $B_V$ denote the Bergman projections of $\Omega$ and $V$, respectively.  Assume that $B_\Omega$ is such that $B_\Omega(P)=HP.$  Then, if there exists a polynomial biholomorphic mapping $f: \Omega \rightarrow V$ with polynomial inverse, it follows that $B_V (P) = HP$. 
  \end{thm}
  \begin{proof}
  Let $p\in P$ be any polynomial function on V.  We will show that $B_Vp \in HP$.  As usual, $HP \subset B_V(P)$ is clear.
  
  Using the transformation formula (\ref{bergtrans}) for $f$ as in the statement of the theorem, we have
  \begin{equation}
  \label{trans}
   B_\Omega (u \cdot p \circ f)=u \cdot (B_Vp)\circ f, 
   \end{equation}
  where $u$ is the complex Jacobian determinant of $f$. Let $F=f^{-1}:\thinspace V \rightarrow \Omega$ be the inverse mapping to $f$, and let $U$ be the complex Jacobian determinant of $F$.  By hypothesis, each of $f,\thinspace F$ are polynomial mappings, and so $u$ is a polynomial function on $\Omega$ and $U$ is a polynomial function on $V$.  By the chain rule, since $f\circ F$ is the identity, we have 
  \begin{equation}
  \label{chain}
  (u\circ F)\cdot U \equiv 1
  \end{equation}
  as functions on $V$.
  
  Now rearrange (\ref{trans}) by dividing by $u$ and composing on the right with $F$ on each side.  Using (\ref{chain}), this yields that
  \begin{equation}
  \label{BV}
  B_Vp=U \cdot B_\Omega(u \cdot p \circ f)\circ F
  \end{equation}
  Now, since $u, \thinspace p, \thinspace f$ are all polynomial, the function $u \cdot p \circ f$ is a polynomial function on $\Omega$, so by hypothesis $B_\Omega(u \cdot p \circ f)$ is a polynomial function on $\Omega$.  Since $F$ and $U$ are polynomial functions on $V$, it follows immediately that $U \cdot B_\Omega(u\cdot p \circ f)\circ F$ is a polynomial function on $V$.  Hence $B_Vp \in HP$.
  \end{proof}
  
  We remark that in one dimension, the only biholomorphic polynomial mappings with polynomial inverse are of degree 1, but many other such mappings exist in dimensions greater than 1. We can use Theorem \ref{bihol} to find domains which are neither ellipsoids nor bounded circular domains on which $B(P)=HP$.
  
  As an explicit example, let $\Omega$ be the unit polydisc in $\mathbb{C}^2$, \[ \Omega = \{(z_1,z_2)\in \mathbb{C}^2 \thinspace : \thinspace |z_1|<1 \thinspace \text{and} \thinspace |z_2|<1 \},\]
  and let $f: \thinspace \mathbb{C}^2 \rightarrow \mathbb{C}^2$ be the polynomial mapping defined by 
  \[f(z_1,z_2)=(z_1+z_2^2,z_2).\]  It is easy to verify that $f$ is univalent on all of $\mathbb{C}^2$, with inverse $F$ given by \[F(\zeta_1,\zeta_2)=(\zeta_1-\zeta_2^2,\zeta_2).\] Define $V=f(\Omega)$, and apply Theorem \ref{bihol} with $f$ restricted to $\Omega$ and $F$ restricted to $V$ see that the projection $B_V$ is such that $B_V(P)=HP$.  The domain $V$ is neither a circular domain about any point of $\Omega$, nor an ellipsoid.  To see this is a matter of a few simple calculations.
  
  First, note that the points $(\frac{91}{100}, \frac{1}{10})$ and $(\frac{171}{100},\frac{9}{10})$ are both members of $V$ (being the images under $f$ of the points $(\frac{9}{10}, \frac{1}{10})$ and $(\frac{9}{10},\frac{9}{10})$, respectively). Their midpoint is $M=(\frac{131}{100},\frac{1}{2}),$ which is not a member of $V$, since $F(M)=(\frac{53}{50},\frac{1}{2})$, which lies outside $\Omega$.  Hence $V$ is not convex, and therefore cannot be an ellipsoid.
  
  Second, consider again the point $(\frac{171}{100},\frac{9}{10})$ of $V$.  If $V$ were a circular domain about the origin, then the point $(-\frac{171}{100},-\frac{9}{10})$ would also be a member of $V$, but applying $F$ to this point we find $F((-\frac{171}{100},-\frac{9}{10}))=(-\frac{63}{25}, -\frac{9}{10})$, which is not an element of $\Omega,$ so $(-\frac{171}{100},-\frac{9}{10})$ is not a member of $V$, and $V$ fails to be circular about the origin. 
  
   If $V$ were circular about the point $a\in V$, then by \cite{Bell2} the Bergman kernel at $a$, $K_V(\zeta,\thinspace a)$ would be constant in $\zeta \in V$.  Since the Jacobian determinant of $F$ is the function $1$, the transformation formula for the Bergman kernel under biholomorphisms yields that $K_\Omega(z,\thinspace F(a))$ is constant in $z \in \Omega$.  However, since $\Omega$ is itself circular about the origin, \cite{Bell2} gives that $K_\Omega(z,\thinspace 0)$, the Bergman kernel at the origin, is also constant.  Hence, by the reproducing property of the Bergman kernel, there exists a complex constant $\lambda$ such that $h(0)=\lambda h(F(a))$ for all $h \in H^2(\Omega)$.  Now, this is only possible if $\lambda=1$ and $F(a)=0$.  Applying $f$, we have that $a=f(0)=0$.  Thus $V$ is circular about the origin, but this possibility was excluded in the previous paragraph. 
   
   \section{A hierarchy of Khavinson-Shapiro Conjectures}
   
   As a final consideration, we can employ another variation of the linear algebra proof from Section $3$ to show that on ellipsoids in $\mathbb{R}^n$, the orthogonal projection from $L^2$ real-valued functions onto its subspace of polyharmonic functions of order $m$, which we call the `Bergman Projection onto polyharmonic functions of order $m$', takes real polynomials to real polynomials without increasing degree. (The projection is defined since the space of polyharmonic functions of order $m$ is closed in $L^2$, for example by hypoellipticity of the operator $\Delta ^m$).  Recall that, for positive integers $m$, the polyharmonic functions of order $m$ on a domain are those functions $f$ such that $\Delta ^m f=0$. We mention that the Khavinson-Shapiro conjecture for the polyharmonic Dirichlet problem has recently been established for a particular class of domains in \cite{R} (cf Sec. 10, Thm 31).
   
    Given an ellipsoid $\mathcal{E}\subset \mathbb{R}^n$, let the Bergman projection onto polyharmonic functions of order $m$ be denoted by $B^{(m)}$, let $\mathcal{H}^{(m)}_N$ denote the space of polynomials which are polyharmonic of order $m$ and of total degree at most $N$. Let $P_N$ be the space of real polynomials of degree at most N, and let $r$ be a degree-2 defining polynomial for $\mathcal{E}$.
  
  Mimicking the proof of Theorem \ref{main}, let $\varphi$ be the linear map from $P_N$ to itself defined by
  \[ \varphi (p)=\Delta ^m r^{2m} \Delta^m p. \]
  Now, since $\varphi$ clearly vanishes on $\mathcal{H}^{(m)}_N$, it descends to a map $\tilde{\varphi}$ from $P_N/ \mathcal{H}^{(m)}_N$ to itself, where $P_N$ is the space of all polynomials of degree at most $N$. Considering $P_N / \mathcal{H}^{(m)}_N$ as a real vector space, $\tilde{\varphi}$ is linear, and can be shown to be injective.  Indeed, suppose that $\tilde{\varphi}(p)=[0]$.  Then ${\varphi}(p) \in \mathcal{H}^{(m)}_N$.  But $\varphi (p)$ is also orthogonal to $\mathcal{H}^{(m)}_N$ by integration by parts. To see this, let $q \in \mathcal{H}^{(m)}_N$, and,  using the usual inner product on real-valued functions, notice that
  
  \[ \int_{\mathcal{E}}{\Delta^m r^{2m} \Delta^m p \cdot q } = \int_{\mathcal{E}}{r^{2m} \Delta ^m p \cdot \Delta^m q} = \int_{\mathcal{E}}{r^{2m} \Delta ^m p \cdot 0}=0. \]
  
  The use of the self-adjointness of $\Delta$ has been employed $m$ times, and this is justified since $r^{2m}$ along with all of its partial derivatives of order up to $2m-1$ vanish on the boundary of $\mathcal{E}$.
  
  Hence ${\varphi}(p)=0.$  Write this as
   \[ \Delta \Delta^{m-1} r^{2m} \Delta ^m p=0, \] and notice that by counting derivatives, 
   $ \Delta^{m-1} r^{2m} \Delta ^m p$
      vanishes on the boundary of $\mathcal{E}$. Hence by the maximum principle for harmonic functions,
       \[ \Delta^{m-1} r^{2m} \Delta ^m p =0 .\]
         Repeating this argument $m-1$ more times, we see that $r^{2m} \Delta^m p =0$, and since $r$ is nonvanishing on $\mathcal{E}$, we have $\Delta ^m p = 0$, so $[p]=0$.

So $\tilde{\varphi}$ is injective, and it is also surjective since it is a linear map between vector spaces of equal finite dimension.  Given a polynomial $P \in P_N$, let $Q \in P_N$ be such that $[P]=\tilde{\varphi}([Q]).$  There exists $h \in \mathcal{H}^{(m)}_N$ such that 
$P=\varphi(Q) + h.$  But by the integration by parts argument above applied to $Q$, $\varphi(Q)$ is orthogonal to $\mathcal{H}^{(m)}_N$.  So we in fact have found an orthogonal decomposition, and $B^{(m)}P =h$ is a polynomial of degree at most $N$. We have proved the following theorem.

\begin{thm}
\label{Harmonics}
Let $\mathcal{E} \subset \mathbb{R}^n$ be an ellipsoid. Given any positive integer $m$, let $B^{(m)}$ be the Bergman projection from $L^2(\mathcal{E})$ onto polyharmonic functions of order $m$. Then, for each positive integer $N$, $B^{(m)}(P_N)= \mathcal{H}^{(m)}_N$.
\end{thm}

In essence, this means we have conceived a whole hierarchy of Khavinson-Shapiro conjectures, one for each of the possible values of $m$.
 We remark that we can explicitly relate Theorem \ref{Harmonics} to potential theory by noticing that for smooth bounded domains, $B^{(m)}=I- \Delta ^m G^{(2m)} \Delta ^m$ for each positive integer $m$.  Here $I$ is the identity operator and $G^{(2m)}$ is the solution operator for the polyharmonic Dirichlet problem $\Delta ^{2m} \varphi = v$, $\varphi=D_j \varphi =0$ on $bd\mathcal{E}$, where $D_j$ stands for every partial derivative of order at most $m-1$.  For an ellipsoid $\mathcal{E}$, whenever $p$ is a polynomial of degree $N$, it follows that $\Delta ^m G^{(2m)} \Delta ^m p$ is a polynomial of degree at most $N$, and we have a formulation in terms of solutions to a PDE, in comparison with the original formulation of the Khavinson-Shapiro conjecture.

 {\bf Acknowledgements.} In closing, the author would like to thank Steve Bell and Erik Lundberg for insightful discussions regarding the content of this work.

\end{document}